\documentclass[12pt,letterpaper]{amsart}
\usepackage{amsmath,amsthm,amsfonts,amssymb,amscd}
\usepackage{lastpage}
\usepackage{enumerate}
\usepackage{fancyhdr}
\usepackage{mathrsfs}
\usepackage{xcolor}
\usepackage{graphicx}
\usepackage{listings}
\usepackage{hyperref}
\usepackage{enumitem}
\usepackage{tikz}
\usepackage{comment}

\usepackage[colorinlistoftodos]{todonotes}

\usepackage{parskip}
\newtheorem{thm}{Theorem}[section]

\newtheorem*{introthm*}{Theorem}

\newtheorem{cor}[thm]{Corollary}

\newtheorem*{lem*}{Lemma}
\newtheorem{prop}[thm]{Proposition}

\theoremstyle{definition}

\newtheorem{claim}[thm]{Claim}

\theoremstyle{remark}
\newtheorem{rem}[thm]{Remark}

\newtheorem*{notation*}{Notation}

\renewcommand{\P}{{\mathbb P}}

\newcommand{\C}{{\mathbb C}}

\newcommand{\ii}{{\mathrm i}}

\DeclareMathOperator{\Aut}{Aut}

\DeclareMathOperator{\Sym}{Sym}

\author{Lev Borisov}
\address{Hill Center Department of Mathematics, Rutgers University, NJ 08854}
\email{borisov@math.rutgers.edu}
\urladdr{https://sites.math.rutgers.edu/\~{}borisov}

\author{Mattie Ji}
\address{Brown University, Department of Mathematics, Box 1917, 151 Thayer Street, Providence, RI 02912, USA}
\email{mattie\_ji@brown.edu}

\author{Yanxin Li}
\address{Hill Center Department of Mathematics, Rutgers University, NJ 08854}
\email{yl1203@scarletmail.rutgers.edu}

\title[Explicit equations of $(a=7,p=2,\emptyset,D_3 X_7)$]{Explicit equations of the fake projective plane $(a=7,p=2,\emptyset,D_3 X_7)$}

\begin{document}

\begin{abstract}
We find explicit equations of the fake projective plane $(a=7,p=2,\emptyset,D_3 X_7)$, which lies in the same class as the fake projective plane    
$(a=7,p=2,\emptyset,D_3 2_7)$ with $21$ automorphisms whose equations were previously found by Borisov and Keum. The method involves finding a birational model of a common Galois cover of these two surfaces.
\end{abstract}

\maketitle

\section{Introduction}
Fake projective planes (FPPs for short) are defined as complex algebraic surfaces of general type with Hodge numbers equal to that of $\C\P^2$.
The first example of an FPP was constructed by Mumford in \cite{Mumford} using the method of $2$-adic uniformization. Efforts of multiple researchers over the ensuing decades  resulted in a complete classification of these surfaces as free quotients of the complex $2$-ball
$$
\mathcal B^2= \{|z_1|^2 + |z_2|^2<1\}\subset \C^2
$$
by explicit discrete subgroups of $\operatorname{PU}(2,1)$. We refer the reader to \cite{BK19} for the history of the field.

\smallskip
However, if one is given a fake projective plane $\mathcal B^2/\Gamma$ with explicit generators and relations of $\Gamma$, it is still a highly nontrivial problem to write explicit equations of $\mathcal B^2/\Gamma$ inside some projective space. The issue here is that there is no known algorithm for calculating modular forms in a way that would allow one to find polynomial equations among them.  An extensive program of finding equations of fake projective planes, initiated in \cite{BK19} and continued in \cite{BF20,BBF20,Aut21,ZL, BJLM}, has up until now produced explicit equations for $10$ out of $50$ conjugate pairs of FPPs. This paper describes a method of obtaining one more pair, bringing the count up to $11$. As is in all other efforts of finding FPPs, heavy computer calculations are necessary. We provide the code and the results in the supplementary materials \cite{BJL23+}. Our primary tool is \texttt{Mathematica} \cite{Ma}, with some computations performed in \texttt{Magma} \cite{magma} and \texttt{Macaulay2} \cite{M2}.

\smallskip
The fake projective plane $X = (a=7,p=2,\emptyset,D_3 X_7)$ lies in the same class as the fake projective plane    
$Y=(a=7,p=2,\emptyset,D_3 2_7)$ which implies that there is a surface $Z$ which is an \'etale Galois cover of both $X$ and $Y$.
Our construction hinges on the observation that in this particular case the Galois groups in question are the dihedral group $D_{14}$ of $14$ elements and the cyclic group $C_{14}$. More precisely, there exists the following diagram of maps of surfaces.
\begin{equation}
Y \stackrel {\mu}{\longleftarrow} W \stackrel {\pi}{\longleftarrow} Z \stackrel {\rho}{\longrightarrow} X
\end{equation}
where 
\begin{itemize}
\item $\mu$ is an unramified cyclic $C_2$ quotient map;
\item $\pi$ is an unramified cyclic $C_7$ quotient map;
\item $\mu\circ\pi$ is an unramified  quotient map by the group isomorphic to $D_{14}$;
\item $\rho$ is an unramified cyclic $C_{14}$ quotient map.
\end{itemize}
The above is, of course, a direct consequence of the relations among the fundamental groups of these surfaces, which we denote by $\Gamma_X,\ldots,\Gamma_W$. Note that all of them are finite index subgroups in the maximal arithmetic subgroup $\bar\Gamma$ computed in \cite{CS11+}.

\smallskip
The paper is organized as follows. In Section~\ref{galoiscovers} we study the relationship between the subgroups $\Gamma_X,\Gamma_Y,\Gamma_Z,\Gamma_W$ of $\bar\Gamma$ which is key to our approach. In Section \ref{secZ} we describe how we obtain an embedding of $W$ in $\C\P^{19}$ and a birational model of $Z$ in $\C\P^{12}$. In Section \ref{secX} we describe the method which allowed us to pass from $Z$ to an embedding of $X$ into $\C\P^9$.
Finally, in Section \ref{comments} we make a couple of comments regarding possible further directions of this project.

{\bf Acknowledgements}. This project sprang from the REU experience organized as part of the larger DIMACS REU at Rutgers University. This research was conducted while the second author was participating in the 2023 DIMACS REU program at the Rutgers University Department of Mathematics, mentored by the first author. We are grateful to Lazaros Gallos, Kristen Hendricks, and Rutgers University DIMACS for organizing the REU program and making this project possible. We also thank Sargam Mondal for useful questions and comments.

\section{Galois covers}\label{galoiscovers}
As was pointed out in the Introduction, the fake projective plane $X= (a=7,p=2,\emptyset,D_3 X_7)$ is in the same class (in the sense of \cite{PY,PY2, CS10}) as the 
well-understood fake projective plane $Y=  (a=7,p=2,\emptyset,D_3 2_7)$ whose equations were constructed in \cite{BK19}. This means that the corresponding subgroups $\Gamma_X$ and $\Gamma_Y$ of $\operatorname{PU}(2,1)$ are finite index subgroups in some arithmetic group $\bar\Gamma$. In our case, $\Gamma_Y$ is normal while $\Gamma_X$ is not, and both have index $21$, see \cite{CS11+}.

\smallskip
While any two finite index subgroups contain a common normal subgroup of finite index, this index can in general be quite large. Fortunately, in our case, we found a subgroup $\Gamma_Z\subset \bar\Gamma$ which is a normal subgroup of index $14$ in both $\Gamma_X$ and $\Gamma_Y$, so that $Z = \mathcal B^2/\Gamma_Z$ is an unramified Galois cover of both $X$ and $Y$. 

\begin{prop}\label{group294}
The commutator subgroup $\Gamma_Z =(\Gamma_X)'$ is a normal subgroup of $\bar\Gamma$. Moreover, $\Gamma_Z$ is contained in both $\Gamma_X$ and $\Gamma_Y$. The quotient group $\bar\Gamma/\Gamma_Z$ has order $294$ and is isomorphic to a certain semidirect product $(D_{14}\times C_7) \rtimes C_3$ with generators
$(t_1,t_2,t_3,t_4)$ and relations
\begin{align*}
1=t_1^2=t_2^7=(t_1 t_2)^2 = t_3^7,~t_1t_3=t_3t_1, t_2 t_3 =t_3 t_2, \\
1=t_4^3,~t_4 t_1 t_4^{-1} = t_1, ~t_4 t_2 t_4^{-1} = t_2^4, ~
t_4 t_3 t_4^{-1} = t_3^2. 
\end{align*}
Here  $D_{14}$ denotes the dihedral group of $14$ elements generated by $t_1$ and $t_2$ and the action of $C_3$ is seen in the above relations.
\end{prop}

\begin{proof}
The group $\bar\Gamma$ is computed in \cite[bargammapresentations.txt]{CS11+} as the finitely presented group with generators $z$ and $b$ and relations
\begin{align*}
z^7,
(b^{-2}z)^3,
(b^2z^{-2}b^2z^2)^3,
(b^2z^{-2}b^2z^4)^3,
b^3z^{-2}b^{-1}z^2b^{-2}z,\\
b^3 zb^3z^3bz^2b^{-1}z^{-1},
b^3z^2b^2z^{-2}b^{-1}z^{-1}b^{-3}zb^{-1}z^{-1}.
\end{align*}
The subgroups $\Gamma_X$ and $\Gamma_Y$ are generated by
$$
b^3,
zb^3z,
bz^2b^{-1}z
$$
and
$$
b^3, (zbz^{-1})^3,
b z b^2z^{-2},
z b z^3b^{-1}
$$
respectively.

\smallskip
\texttt{Magma} \cite{magma} readily computes \cite[Section2.txt]{BJL23+} that $\Gamma_Z=(\Gamma_X)'$ is a normal subgroup of $\Gamma_Y$ and $\bar\Gamma$. After that, with some trial and error, we found generators 
of $\bar\Gamma/\Gamma_Z$ (images of) 
$$
t_1=b^3,~
t_2=bzb^2z^{-2}b^3,~
t_3=bz^2b^{-1}z,~
t_4=b^4
$$
which satisfy the relations of the proposition. See \cite[Section2.txt]{BJL23+} for the \texttt{Magma} verification.
\end{proof}

From now on we will tacitly identify the quotient $\bar\Gamma/\Gamma_Z$ with the above semidirect product.
\begin{prop}
The quotients $\Gamma_X/\Gamma_Z$ and $\Gamma_Y/\Gamma_Z$ are given by $(C_2\times C_7)\rtimes\{1\}$ and  $(D_{14}\times \{1\})\rtimes\{1\}$
respectively, for an appropriate choice of $C_2\subset D_{14}$.  
\end{prop}

\begin{proof}
We verify in \cite[Section2.txt]{BJL23+} that the quotients $\Gamma_X/\Gamma_Z$ and $\Gamma_Y/\Gamma_Z$ are generated by 
$\{t_1,t_3\}$ and $\{t_1,t_2\}$ respectively.
\end{proof}

\begin{cor}
The surfaces $X$ and $Y$ are given by  $X = Z/C_{14}$ and $Y = Z/D_{14}$. 
\end{cor}

\begin{rem}\label{autY}
The elements $t_3$ and $t_4$ generate the quotient group $\bar\Gamma/\Gamma_Y$, isomorphic to $C_7\rtimes C_3$. It is the automorphism group of $Y$ and it acts on the homogeneous coordinates of the embedding of $Y$ into $\C\P^9$ by scaling and/or permuting the coordinates, see \cite{BK19}.
\end{rem}

\smallskip
We will also be interested in the intermediate surface $W= Z/C_7$ with the action of $C_7$   coming from the normal subgroup of $D_{14}$ in $\bar\Gamma/\Gamma_Z$. The corresponding group $\Gamma_W$ is an index two subgroup of $\Gamma_Y$ generated by $\Gamma_Z$ and $t_2$. The surface $W$ is the unramified double cover of $Y$ which corresponds to the unique nontrivial automorphism-invariant $2$-torsion line bundle on $Y$, see \cite{BK19,BJLM}. Indeed, both $\Gamma_W$ and $\Gamma_Y$ are normal in $\bar\Gamma$ (see \cite[Section2.txt]{BJL23+}) and thus the index two subgroup $\Gamma_W\subset \Gamma_Y$ is preserved by the conjugation action of $\bar\Gamma/\Gamma_Y$. 

\smallskip
In what follows we will need the following.
\begin{prop}
The Hodge numbers of $Z$ are 
$h^{1,0}(Z) = 0$, $h^{1,1}(Z)=14$ and $h^{2,0}(Z)=13$.
\end{prop}

\begin{proof}
The Euler characteristics of $\mathcal O_Z$ is $14$, since $Z\to Y$ is unramified. We used the \texttt{Magma} file \cite[Section2.txt]{BJL23+} to check that $h^{1,0}(Z)=0$ since the abelianization of $Z$ is finite. This implies that $h^{2,0}(Z)=13$, and $h^{1,1}(Z)=14$ follows from $\chi_{top}(Z) = 14 \chi_{top}(Y) = 42$.
\end{proof}

Moreover, we can identify the representation of $\bar\Gamma/\Gamma_Z$ on the $13$-dimensional space $H^0(Z,K_Z)$. Let us first discuss the action of $D_{14}\times C_7$. Recall that $D_{14}$ has two one-dimensional representations, which we will call $V_+$ and $V_{-}$ coming from $C_2=D_{14}/C_7$ and three two-dimensional representations which we will call $V_{1}$,  $V_{2}$ and $V_4$.
Specifically, the generator of $C_7\subset D_{14}$ has eigenvalues $\zeta_7^{\pm i}$ on $V_i$, and there is a basis of $C_7$ eigenvectors of $V_i$ on which $C_2\subset D_{14}$ acts as a permutation. Irreducible representations of $D_{14}\times C_7$ are then tensor products of the above representations of $D_{14}$ and a one-dimensional representation of $C_7$, and we denote them accordingly. For example, $V_{2,4}$ stands for the tensor product of $V_2$ with the representation of $C_7$ that sends its generator to $\zeta_7^4$. It has a basis $(r_{2,4},r_{-2,4})$ where subscripts denote the weights modulo $7$ with respect to the action of $C_7\times C_7$.

\begin{prop}\label{content}
As a representation of $D_{14}\times C_7$, after a choice of a generator of $\{1\}\times C_7$, the space $H^0(Z,K_Z)$ is isomorphic to a direct sum of representations 
$$
V_{-,0}\oplus V_{1,1} \oplus V_{4,2} \oplus V_{2,4} \oplus V_{1,a} \oplus V_{4,2a} \oplus V_{2,4a}
$$
for some $a\in \{3,5,6\}\hskip -5pt\mod 7$. 
\end{prop}

\begin{proof}
Since the actions of $D_{14}$ and $C_{14}$ have no fixed points and $H^1(Z,K_Z)=0$, the holomorphic Lefschetz fixed point formula implies that $H^0(Z,K_Z)\oplus H^2(Z,K_Z)$ is a regular representation for both of these groups. Since $H^2(Z,K_Z)$ is a trivial one-dimensional representation, we see that 
$H^0(Z,K_Z)$ is isomorphic to
$$
V_{-,0}\oplus V_{1,a_1} \oplus V_{1,a_2} \oplus  V_{2,a_3} \oplus V_{2,a_4} \oplus  V_{4,a_5} \oplus V_{4,a_6}
$$
with $(a_1,\ldots,a_6)$ a permutation of $(1,\ldots 6)$. 

\smallskip
By picking a generator of $\{1\}\times C_7$, we can assume without loss of generality that the first two 
$V$-s are $V_{1,1}$ and $V_{1,a}$ for some $a\not\in \{0,1\}\mod 7$. 
Furthermore, we have the action of the group $C_3$ which conjugates $V_{a,b}$ into $V_{4a,2b}$ because of Proposition \ref{group294}. This uniquely determines the rest of the $V$-s. It remains to observe that $a$ can not be $2$ or $4$ either because it would contradict the above permutation property.
\end{proof}

\begin{rem}
Since we can choose a generator of $\{1\}\times C_7$ so that $V_{1,a}$ becomes $V_{1,1}$ and $V_{1,1}$ becomes $V_{1,a^{-1}}$, we can further reduce to $a=3$ or $a=6$ cases. We will later see that $a=3$ case works. In principle, we could have used more extensive \texttt{Magma} calculations and the holomorphic Lefschetz formula to fully specify the representation of $\Aut(Z)$ on $H^0(Z,K_Z)$, but we find it easier to use this approach, which will be enough for our purposes.
\end{rem}

\begin{cor}\label{r-s}
There is a basis
$$
(r_{0,0}, r_{1,1},r_{-1,1},r_{4,2},r_{-4,2},r_{2,4},r_{-2,4},r_{1,a},r_{-1,a},r_{4,2a},r_{-4,2a},r_{2,4a},r_{-2,4a})
$$
of $H^0(Z,K_Z)$, indexed by subscripts in $(C_7)^2$, with the following action of the automorphism group of $Z$:
\begin{itemize}
\item the generator of $C_7\times \{1\}$ sends $r_{i,j}$ to $\zeta_7^i r_{i,j}$;
\item the generator of $\{1\}\times C_7$ sends $r_{i,j}$ to $\zeta_7^j r_{i,j}$;
\item the generator of $C_2\subset  D_{14}$ sends $r_{0,0}$ to $(-r_{0,0})$ and interchanges $r_{i,j}$ and $r_{-i,j}$;
\item the generator of $C_3$ sends $r_{i,j}$ to $r_{4i,2j}$.
\end{itemize}
\end{cor}

\begin{proof}
Follows immediately from Proposition \ref{content}.
\end{proof}

\section{Finding explicit equations of $Z$}\label{secZ}
In this section, we describe the rather lengthy process that allowed us to construct a birational model of $Z$ as a repeated cyclic cover of $Y$.

\smallskip
In the first step, we constructed the double cover $W$ of $Y$ as a surface in its bicanonical embedding into $\C\P^{19}$, cut out by $100$ quadratic equations. While it was already accessible via methods of \cite{BK19}, we used a different approach. We note that $W\to Y$ corresponds to the $2$-torsion divisor $D$ on $Y$ and $H^0(W,2K_W)$ can be naturally identified with $H^0(Y,2K_Y)\oplus H^0(Y,2K_Y+D)$. The action of $\Aut(Y)$ extends to 
$H^0(Y,2K_Y+D)$ and there is exactly one invariant section $s$ on it, which we find as follows.

\smallskip
Recall that there is a unique ample divisor class $H$ on $Y$ such that $3H=K_Y$. In \cite{BJLM}, we studied linear systems $|nH+T|$ for small $n$ and torsion divisors $T$. Specifically, $|4H+D|$ has three $C_7$-equivariant sections that are permuted by $C_3$, and we know explicit equations of these curves on $Y$. This allows us to find the curve $\{s=0\}$ on $Y$ by looking for an automorphism-invariant section of $6K_Y = 18H$ that vanishes on
the three aforementioned curves in $|4H+D|$ linear system and using
$$
6H+D = 18 H - 3 ( 4H+D).
$$
Once we have found this curve $\{s=0\}$, we are able to find an $\Aut(Y)$-invariant section of $12H=2(6H+D)$ which vanishes twice on it, i.e. it is equal to $s^{\otimes 2}$, up to scaling, which we denote by $U_{10}$. We can then use $s$ to embed sections of $2K_Y+D$ into $H^0(Y, 4K_Y)$. Note that the map of $55$-dimensional spaces $\Sym^2(H^0(Y,2K_Y))\to H^0(Y, 4K_Y)$ is an isomorphism, and we can view sections of $4K_Y$ as quadratic polynomials in the defining variables $U_0,\ldots, U_9$ of $Y$ in \cite{BK19}. The image of $\otimes s: H^0(2K_Y+D)\to H^0(4K_Y)$ is the ten-dimensional space of quadratic polynomials in $U_0,\ldots,U_9$ which are zero on $\{s=0\}$. We pick a basis of it with a nice action of $\Aut(Y)$ (the $C_7$ acts by scaling the basis elements and $C_3$ acts by permuting them) and denote it by $U_{10},\ldots, U_{19}$.  

\smallskip
We constructed the basis  $(P_0,\ldots,P_{19})$ of $H^0(W,2K_W)$ by looking at $P_i = U_i$ for $i=0,\ldots,9$ and $P_i = \frac {U_i}{\sqrt {U_{10}}}$ for $i=10,\ldots,19$. We then picked random points on $Y$ to find quadratic relations on $P_i $. Note that the covering involution of $W\to Y$ acts by $(+1)$ on $P_0,\ldots,P_9$ and by $(-1)$ on $P_{10},\ldots,P_{19}$. We also had to be careful with the scaling of $U_{10}$ to ensure that our model of $W$ is defined over the field ${\mathbb Q}(\sqrt{-7})$. More precisely, we used
\begin{align*}
U_{10} =& \frac{(49+13 \ii \sqrt{7})}{56} \Big(\frac1{64} (-17-7 \ii \sqrt{7}) U_0^2+U_1 U_4+U_2 U_5+U_3 U_6
\\&
+\frac 18 (-1+\ii \sqrt{7}) U_1 U_7+
\frac 18 (-1+\ii \sqrt{7}) U_2 U_8+\frac 18 (-1+\ii \sqrt{7}) U_3 U_9\Big),
\end{align*}
see \cite[Section3.nb]{BJL23+}.

\begin{claim}
The double cover $W$ of $Y$ is cut out by $100$ quadratic equations from \cite[W\_equations.txt]{BJL23+}. 
\end{claim}

\begin{rem}
We did not try to do all of the formal checks on $W$, as the number of variables is likely too high, although we did check its Hilbert polynomial in \cite[W\_check\_Hilbert.txt]{BJL23+}.
\end{rem}

\begin{rem}
We choose to make a distinction between statements that are verified by hand or by computer using symbolic manipulation, as opposed to the statements that are derived by picking some random points on the varieties in question, as is often the case in our \texttt{Mathematica} calculations. We call the former propositions, lemmas, et cetera, and use the term ``claim" to indicate the latter. 
\end{rem}


The next step (passing from $W$ to $Z$) is arguably the most technical in the entire process. It is similar to the method of \cite{BF20}. Let $r_{1,1},r_{-1,1},r_{1,a},r_{-1,a}$ be sections of $H^0(Z,K_Z)$ as in Corollary \ref{r-s}. Consider 
$$
 s_1 = r_{1,1} r_{-1,a}, s_2 = r_{-1,1} r_{1,a}, s_3 = r_{1,1} r_{-1,1}, s_4 = r_{1,a} r_{-1,a}.
$$
These $s_i$ are invariant with respect to $C_7\times \{1\}$ and therefore are pullbacks of sections of $H^0(W,2K_W)$. They must satisfy a number of restrictions.
\begin{itemize}
\item There holds $s_1 s_2 = s_3 s_4$.
\item Sections
$s_1$ and $s_2$ have weight $(a+1)$ with respect to the $\{1\} \times C_7$ action on $W$.
\item Sections $s_3$ and $s_4$ have weight $2$ and $2a$ respectively under the above action.
\item The covering involution of $W\to Y$ switches $s_1$ and $s_2$ and preserves $s_3$ and $s_4$. 
\end{itemize}
This information turns out to be sufficient to reconstruct 
$s_i$, up to scaling.

\begin{claim}
For the choice of $a=3$ we can pick
{\footnotesize
\begin{align*}
s_1= &-P_6 -\frac 14 P_9 -\frac 14(1+{\rm i} \sqrt 7)P_{16}+ P_{19} ,~
s_2 = - P_6 -\frac 14 P_9+ \frac14 (1+ {\rm i} \sqrt 7)P_{16} -P_{19},
\\
s_3 =& P_5,~~
s_4= \frac 18 (5+ {\rm i } \sqrt 7)P_1.
\end{align*}
}
\hskip -5pt
Here $P_i$ are the coordinates on $\C\P^{19}$ in which $W$ is embedded. There are no solutions for $a=6$.
\end{claim}

\begin{rem}
If $\pi:Z\to W$ is a quotient map, we have 
$$\pi_*\mathcal O_Z \cong \bigoplus_{i=0}^6  T^{\otimes i}$$
where $T$ is a $7$-torsion line bundle on $W$. The above is also precisely the splitting of  $\pi_*\mathcal O_Z$ into eigenbundles of the $C_7$-action. Since $r_{i,j}$ are eigenfunctions of the $C_7\times \{1\}$ action on $K_Z$, they can be viewed as sections of $\mathcal O_W(3H)\otimes T^{\otimes i}$ and this defines curves $\{r_{i,j}=0\}$ in $W$. Note that while we have not yet really constructed $r_{i,j}$, we can indeed construct these curves. For instance, $r_{1,1}=0$ is the set of common zeros of $s_1$ and $s_3$ on $W$.
\end{rem}

Once all the $s_i$ are found, there are multiple approaches to finding a birational model of $Z$. We chose an approach that manifestly preserved the $C_3$ action, in addition to the $D_{14}\times C_7$ action. Specifically, the field of rational functions of $Z$ is a degree $7$ cyclic extension of the field of rational functions of  $W$, and we can pick the basis of it to be 
\begin{equation}\label{basis}
(1, \frac  {r_{1,1}}{r_{0,0}}, \frac  {r_{4,2}}{r_{0,0}}, \frac  {r_{2,4}}{r_{0,0}}, \frac  {r_{-1,1}}{r_{0,0}}, \frac  {r_{-4,2}}{r_{0,0}}, \frac  {r_{-2,4}}{r_{0,0}})
\end{equation}
since these cover all possible first subscripts. We can then compute the multiplication table and use it to find the polynomial relations among $r_{i,j}$. We take care to keep $C_3$ action throughout.

\smallskip
For example, suppose we want to multiply $ \Big(\frac  {r_{4,2}}{r_{0,0}}\Big)\Big( \frac  {r_{2,4}}{r_{0,0}}\Big)$. We know that the product will be in the $\zeta_7^6$ eigenspace of $C_7\subset D_{14}$, so we have 
$$
 \Big(\frac  {r_{4,2}}{r_{0,0}}\Big)\Big( \frac  {r_{2,4}}{r_{0,0}}\Big) =  F \, \Big(\frac {r_{-1,1}}{r_{0,0}}\Big)
$$
for some rational function $F$ on $W$. Then $F$ has simple zeros at $\{r_{2,4}=0\}$ and $\{r_{4,2}=0\}$ and poles at $\{r_{0,0}=0\}$ and $\{r_{-1,1}=0\}$. This determines $F$ uniquely up to scaling, and it can be found explicitly as follows. We know that the divisor of $s_3$ is $\{r_{-1,1}=0\}+\{r_{1,1}=0\}$ and the divisor of $P_{10}$ is $2\{r_{0,0} = 0\}$. Therefore, the divisor of $F s_3 P_{10}$ is
$$
\{r_{2,4}=0\} + \{r_{4,2}= 0\} +\{r_{1,1}=0\} + \{r_{0,0}=0\}
$$
which means that $F s_3 P_{10}$ is a global section of $2K_W$. Such sections are given by quadratic polynomials in $P_0,\ldots,P_{19}$ (we make sure to look at a $C_3$-invariant subset of degree two monomials that span the complement to the space of $100$ relations of $W$) and it is a simple matter to put conditions of vanishing on random points of the above curves to determine $F$. We then use $C_3$ action to populate the multiplication table appropriately.

\smallskip
Similarly, if we want to find the square of, for example, $r_{2,4}$, we see that 
$$
 \Big(\frac  {r_{2,4}}{r_{0,0}}\Big)^2 =  F \, \Big(\frac {r_{4,2}}{r_{0,0}}\Big).
$$
For the $s_3'= r_{4,2} r_{-4,2}$ (a $C_3$-translate of $s_3$) we see that the divisor of $F s_3' P_{10}$ is 
$$
2\{r_{2,4}=0\} +\{r_{-4,2}=0\} + \{r_{0,0}=0\}.
$$
We observe that $F=0$ is singular at  $x\in \{r_{2,4}=0\}$ if and only if the gradient of $F$ vanishes on the tangent space of $x$ in $W$, which again allows us to find $F$ up to scaling.

\smallskip
Once the multiplication table is populated, up to scaling, the associativity property is used to fix scalings (up to some natural ambiguity coming from $C_3$-invariant scaling of $r_{i,j}$). We can then compute (recall that $a=3$) 
$$
\frac  {r_{1,3}}{r_{0,0}}, \frac  {r_{4,6}}{r_{0,0}}, \frac  {r_{2,5}}{r_{0,0}}, \frac  {r_{-1,3}}{r_{0,0}}, \frac  {r_{-4,6}}{r_{0,0}}, \frac  {r_{-2,5}}{r_{0,0}}
$$
in terms of the basis \eqref{basis} and use multiplication table to compute the relations among the $13$ basis elements $r_{i,j}$ of $H^0(Z,K_Z)$. While we do not get an embedding, we still get a birational model of $Z$ in $\C\P^{12}$. The details can be found in \cite[Section3.nb]{BJL23+}. 

\begin{rem}
We used the notation 
\begin{align*}(Z_0,\ldots,Z_{12}) &\\
&\hskip -50pt =(r_{0, 0}, 
   r_{1, 1}, r_{4, 2}, r_{2, 4},
   r_{-1, 1}, r_{-4, 2}, r_{-2, 4},
   r_{1, 3}, r_{4, 6}, r_{2, 5},
   r_{-1, 3}, r_{-4, 6}, r_{-2, 5})
\end{align*}
for the homogeneous coordinates on the birational model of $Z$ which hopefully does not cause confusion with $Z$ itself.   
It is important that $C_3$ symmetry is maintained throughout the computation. We also have control over the action of $C_2\subset D_{14}$,
which negates $r_{0,0}$ and permutes the rest of $r_{i,j}$. Specifically, the action of generators $g_3$ and $g_2$ of the respective cyclic groups is given by
\begin{align*}
g_3(Z_0,\ldots, Z_{12}) &= (Z_0,Z_2,Z_3,Z_1,Z_5,Z_6,Z_4,Z_8,Z_9,Z_7,Z_{11},Z_{12},Z_{10}),
\\
g_2(Z_0,\ldots, Z_{12}) &= (-Z_0,Z_4,Z_5,Z_6,Z_1,Z_2,Z_3,Z_{10},Z_{11},Z_{12},Z_7,Z_8,Z_9).
\end{align*}
\end{rem}

\section{Finding explicit equations of $X$.}\label{secX}
In this section, we explain how we obtained equations of $X$ in its bicanonical embedding. 

\smallskip
Since $X$ is obtained from $Z$ by taking $C_{14}$ quotient (by the $C_2\subset D_{14}$ and $\{1\}\times C_7$) in principle one can take the $C_{14}$-invariant global sections of $2K_Z$ and look for polynomial equations among them. The only difficulty that we had to overcome is the fact that $\Sym^2(H^0(Z,K_Z))\to H^0(Z,2K_Z)$ is not surjective, and we can only get a dimension $7$ subspace of the dimension $10$ space $H^0(W,2K_W)\cong H^0(Z,2K_Z)^{C_{14}}$ this way.  Specifically, these are given by linear combinations of
\begin{align*}
Z_0^2, Z_{12} Z_5 + Z_2 Z_9,  Z_{10} Z_6 +  Z_3 Z_7, 
Z_{11} Z_4 + Z_1 Z_8, 
\\
Z_{10} Z_3 +  Z_6 Z_7, Z_1 Z_{11} + Z_4 Z_8, 
Z_{12} Z_2 +  Z_5 Z_9.
\end{align*}

\smallskip
We resolved this issue in a fairly straightforward way. 
Multiplication by $Z_0$ sends $H^0(Z,2K_Z)^{C_{14}}$ into the space of $C_{7}$-invariant, $C_2$-anti-invariant  sections of $H^0(Z,3Z)$
which are zero on $\{r_{0,0}=0\}\subset Z$. While some of these are simply products of $Z_0$ with a $C_{14}$-invariant quadratic polynomials above, we got three additional basis elements, specifically the $C_3$ orbit of 
\begin{align*}
&Z_1 Z_{10}^2 - \frac 14 (-7 + 3 \ii \sqrt{7}) Z_{10} Z_2 Z_5 - 2 Z_1 Z_3 Z_5 + 
 \frac 12 (1 - \ii \sqrt{7}) Z_{12}^2 Z_6 \\
 &+ 2 Z_2 Z_4 Z_6 - Z_{11} Z_6^2 - Z_1 Z_{10} Z_7 + 
 Z_{10} Z_4 Z_7 + \frac 14 (-7 + 3 \ii \sqrt{7}) Z_2 Z_5 Z_7
 \\& - Z_4 Z_7^2 + Z_3^2 Z_8 - 
\frac 12 (1 - \ii \sqrt{7}) Z_3 Z_9^2,
\end{align*}
and then computed the $84$ cubic equations on them.

\smallskip
The final verifications of the smoothness of the resulting surface and of it being the fake projective plane in bicanonical embedding were performed in \cite[FPP\_M2.txt, FPP\_smoothness\_Hodge.txt]{BJL23+} using \texttt{Magma} \cite{magma} and \texttt{Macaulay2} \cite{M2}. To lower the runtime of our verification codes, we reduced the equations of $X$ modulo $p = 37$ which contains a square root of $-7$ and a non-trivial cubic root of $1$. This allowed us to apply a change of coordinates on the equations of $X$ so that the $C_3$ action is diagonalized. This process eliminates roughly $2/3$ of the monomials used in the original equations. 

\section{Further directions}\label{comments}
Techniques of this paper might allow for the computation of explicit equations of some other fake projective planes which are commensurable to the ones with known equations. We leave this for further REU projects, as this is a suitable training ground for young algebraic geometers. The sticky point is whether we can find a common normal subgroup $\Gamma_Z$ in the two groups $\Gamma_X$ and $\Gamma_Y$ in question that has a solvable quotient $\Gamma_Y/\Gamma_Z$. It is considerably more difficult to construct Galois covers which are not repeated cyclic covers.

\smallskip
Finally, just as a teaser to the reader, we would like to point out that we only know how to match \emph{complex conjugate pairs} of fake projective planes with their explicit embeddings, up to complex conjugation. For example, we are not aware of any method that would allow one to precisely say which choice of  $\sqrt{-7}$ in our equations of $X$ corresponds to the specific choice of the generators of $\Gamma_X$ in \cite{CS11+}.

\end{document}